\documentclass[12pt]{article}

\usepackage{amsmath,amsthm,amsfonts,amssymb,amscd,cite}
\usepackage{graphicx}

\allowdisplaybreaks[4]

\newtheorem {Lemma}{Lemma}[section]
\newtheorem {Theorem} {Theorem}[section]
\newtheorem {Corollary}{Corollary}[section]

\usepackage{fullpage}

\begin{document}

\title{Spectral conditions for graphs in which every edge belongs to a factor}

\author{Jin Cai\footnote{E-mail: jincai@m.scnu.edu.cn}, Bo Zhou\footnote{E-mail: zhoubo@scnu.edu.cn} \\
School of  Mathematical Sciences, South China Normal University,\\
Guangzhou 510631, P.R. China}

\date{}
\maketitle

\begin{abstract}
A factor of a graph is a spanning subgraph.
Spectral  sufficient conditions are provided via spectral radius and signless Laplacian spectral radius for graphs with
(i) a matching of given size (particularly,  $1$-factor) containing any given edge, and (ii)
a star factor with a component isomorphic to stars of order two or three containing any given edge, respectively.  
\end{abstract}

\section{Introduction}

We consider finite, simple and undirected graphs. A factor of a graph is a spanning subgraph.
For integers $a$ and $b$ with $0\le a\le b$, an $[a,b]$-factor is such a factor $F$ such that $a\le d_F(x)\le b$ for every vertex $x$. A $1$-factor is a $[1,1]$-factor.
Denote by
$K_{1,m}$ the star of order $m+1$ (i.e., a complete bipartite graph with partite sizes $1$ and $m$).
For positive integer $k$, let $S(k)=\{K_{1,1}, \dots, K_{1,k}\}$. An $S(k)$-factor is  factor for which  each component is (isomorphic to) one of the stars $K_{1,1}, \dots, K_{1,n}$. A $\{K_{1,1}\}$-factor is a $1$-factor. A star factor is an $S(k)$-factor for some $k$.
The existence of factors with given properties  has received  special attention, see, e.g. \cite{CEK,FK,Ka,Lit,LGH,Ne,ZZ}. For example, a graph with factor containing any given edge is known as to be factor covered.
In recent years,
it of great interest for researchers to find spectral sufficient conditions such that a graph has a given factor, see the survey \cite{FLLO}.
Li and Miao \cite{LM} gave spectral condition that implies a graph has a $\mathcal P_{\ge 2}$-factor (a factor consisting of vertex disjoint paths on at least two vertices) containing any given edge.

Given a graph, we denote by $\rho(G)$ the spectral radius of $G$ and $q(G)$ the signless Laplacian spectral radius of $G$. We are concerned two types of factors with given properties.

One is $1$-factor containing any given edge, for which Little et al. \cite{LGH}
gave a characterization. Though this concept has been extended to different factors that contain a given edge, as far as we know,
there is no spectral sufficient condition for a graph  has a $1$-factor containing any given edge.
Feng et al. \cite{FZ} and Kim et al. \cite{KOSS} gave spectral radius conditions that imply a graph on $n$ vertices has a matching of size (at least) $\frac{n-k}{2}$ for $0\le k\le n$. Instead of considering
$1$-factor containing any given edge directly, we establish spectral  conditions that imply a graph on $n$ vertices has a matching of size (at least) $\frac{n-k}{2}$ for $0\le k\le n$ containing any given edge. We show the following results.

\begin{Theorem}\label{Cai1}
Let $G$ be a graph of order $n\ge  5k+6$, where $n \equiv k \pmod 2$. If $\rho(G)\ge \rho(K_2\vee ((k+1)K_1\cup K_{n-k-3}))$, then $G$ has a matching of size $\frac{n-k}{2}$  containing any given edge unless $G\cong K_2\vee ((k+1)K_1\cup K_{n-k-3})$.
\end{Theorem}

\begin{Theorem}\label{Cai2}
Let $G$ be a graph of order $n\ge  5k+7$, where $n \equiv k \pmod 2$. If $q(G)\ge  q(K_2\vee ((k+1)K_1\cup K_{n-k-3}))$, then $G$ has a matching of size $\frac{n-k}{2}$  containing any given edge unless $G\cong K_2\vee ((k+1)K_1\cup K_{n-k-3})$.
\end{Theorem}

The other is an $S(k)$-factor with a component being  $K_{1,1}$ or $K_{1,2}$ containing any given edge, for which   Chen,  Egawa and  Kano \cite{CEK}
gave a characterization. Also, we establish spectral conditions that imply a graph  has an $S(k)$-factor with a component being  $K_{1,1}$ or $K_{1,2}$ containing any given edge. Our results are as follows.

\begin{Theorem}\label{Cai3}
Let $G$ be a  graph of order $n\ge  \frac{3}{2}k+7$ without isolated vertices. If $\rho(G)\ge \rho(K_2\vee(2K_1 \cup K_{n-4}))$, then $G$ has an $S(k)$-factor in which a component being  $K_{1,1}$ or $K_{1,2}$ containing any given edge unless $G\cong K_2\vee(2K_1 \cup K_{n-4})$.
\end{Theorem}

\begin{Theorem}\label{Cai4}
Let $G$ be a graph of order $n\ge  2k+6$ without isolated vertices. If $q(G)\ge q(K_2\vee(2K_1 \cup K_{n-4}))$, then $G$ has an $S(k)$-factor in which a component being  $K_{1,1}$ or $K_{1,2}$ containing any given edge unless $G\cong K_2\vee(2K_1 \cup K_{n-4})$.
\end{Theorem}

\section{Preliminaries}

For a graph $G$, let $V(G)$ be the vertex set of $G$ and $E(G)$ the edge set of $G$.
For  $v\in V(G)$, the neighborhood  $N_G(v)$ of $v$ is the set of vertices adjacent to $v$ in $G$, and the degree of $v$, denoted by $d_G(v)$, is the number $|N_G(v)|$.
For any $S\subseteq V(G)$, let $G[S]$ be the subgraph of $G$ induced by $S$, and write $G-S=G[V(G)\setminus S]$ if $S\ne V(G)$.
For two vertex disjoint graphs $G_1$ and $G_2$,  $G_1\cup G_2$ denotes the disjoint union of $G_1$ and $G_2$, $G_1\vee G_2$ denotes the join of  $G_1$ and $G_2$, which is obtained from $G_1\cup G_2$
by adding all possible edges between any vertex of $G_1$ and any vertex of $G_2$. For positive integer $k$ and a graph $G$, $kG$ denotes the graphs consisting of $k$ vertex disjoint copies of $G$.
Denote by $K_n$ the complete graph of order $n$.

The spectral radius of a graph $G$ is the largest eigenvalue of the adjacency matrix of $G$,
which is defined as the symmetric matrix $\mathbf{A}(G)=(a_{uv})_{u,v\in V(G)}$, where $a_{uv}=1$ if $u$ and $v$ are adjacent, and $a_{uv}=0$ otherwise. The signless Laplacian spectral radius of a graph $G$
is the largest eigenvalue of its signless Laplacian matrix $\mathbf{Q}(G)=\mathbf{D}(G)+\mathbf{A}(G)$,
where $\mathbf{D}(G)$ is the degree diagonal matrix of $G$.

For a square nonnegative matrix $\mathbf{M}$, let $\lambda(\mathbf{M})$ be its the spectral radius (maximum modulus of the eigenvalues), which is an eigenvalue of $\mathbf{M}$ by the Perron-Frobenius theorem.
In particular, we have $\rho(G)=\lambda(\mathbf{A})$ and $q(G)=\lambda(\mathbf{Q})$ for a graph $G$,

The following lemma, following from  the Perron-Frobenius theorem, is well known.

\begin{Lemma} \label{inequit} 
Let $G$ be a graph and $u$ and $v$ two distinct vertices that are not adjacent in $G$. Then
$\rho(G+uv)\ge \rho(G)$ and $q(G+uv)\ge q(G)$. Both inequalities are strict if
$G+uv$ is connected.
\end{Lemma}

For a graph $G$, let $V(G)=V_1\cup \dots\cup V_s$ be a partition of $V(G)$. For $1\le i<j\le s$, set $\mathbf{B}_{ij}$ denotes  the submatrix of $\mathbf{H}(G)\in \{\mathbf{A}(G), \mathbf{Q}(G)\}$ with rows corresponding to vertices in $V_i$ and columns corresponding to vertices in $V_j$.
The matrix $\mathbf{B}=(b_{ij})$, where $b_{ij}$ equals to the average row sums of $\mathbf{B}_{ij}$,
is called the quotient matrix of $\mathbf{H}(G)$ with respect to the partition $V_1\cup \dots \cup V_s$. Furthermore, if $\mathbf{B}_{ij}$ has constant row sum, then we say $B$ is an equitable quotient matrix (with respect to the above partition of $V(G)$). As an immediate consequence of \cite[Lemma 2.3.1]{BH}, one gets the following lemma.

\begin{Lemma} \label{equit}
If $\mathbf{B}$ is an equitable quotient matrix of $\mathbf{H}(G)\in \{\mathbf{A}(G), \mathbf{Q}(G)\}$, then  the eigenvalues of $\mathbf{B}$ are also eigenvalues of $\mathbf{H}(G)$, and  $\lambda(\mathbf{H}(G)) $ is equal to the largest eigenvalue of $\mathbf{B}$.
\end{Lemma}

Given a connected graph $G$ and  $\alpha=0,1$, we denote by $\lambda_\alpha(G)$ the spectral radius of the matrix $\alpha\mathbf{D}(G)+\mathbf{A}(G)$. By Perron-Frobenius theorem, there is a unique unit positive eigenvector corresponding to $\lambda_\alpha(G)$, which is called the Perron vector.
It is known that if there is an automorphism $\phi$ of $G$ such that $\phi(u)=v$, then the entries of the Perron vector at $u$ and $v$ are equal.
The following is known in \cite{WXH} when $\alpha=0$ and in \cite{HZ} when $\alpha=1$, see also \cite{NP}.

\begin{Lemma} \label{GE} Let $G$ be a connected graph and $u$ and $v$ be two  vertices of $G$.
Let $X$ be the perron vector of $\alpha \mathbf{D}(G)+\mathbf{A}(G)$ with $x_u\ge x_v$, where $\alpha=\{0,1\}$. Suppose that $N_G(v)\setminus (N_G(u)\cup \{u\})\ne \emptyset$. Then for any nonempty $N\subseteq  N_G(v)\setminus (N_G(u)\cup \{u\})$,
\[
\lambda_\alpha(G-\{vw: w\in N\}+\{uw: w\in N\})\ge\lambda_{\alpha}(G).
\]
\end{Lemma}

For the graph $G=K_s\vee (K_{n_1}\cup \dots \cup K_{n_t})$ we call the graph $K_s$ appearing first
the outer copy of $G$, and
the $i$-th graph in
$K_{n_1}\cup \dots \cup K_{n_t}$ the $i$-th inner copy of $G$, where $i=1,\dots, t$.

\begin{Lemma} \label{SP}
For positive integers $s, t, n_1, \dots n_t$ with
$t\ge 2$ and $n_1\le \dots \le n_t$, let $G=K_s\vee (K_{n_1}\cup \dots \cup K_{n_t})$ and $X$ be the Perron vector of $\alpha \mathbf{D}(G)+\mathbf{A}(G)$ with $\alpha=\{0,1\}$ , where for $i=1,\dots, t$, $x_i$ is the entry of $X$ at any vertex of the $i$-th inner copy of $G$. Then $x_i\le x_{i+1}$ for $i=1,\dots, t-1$.
\end{Lemma}

\begin{proof}
Denote by $x_0$ the entry of $X$ at any vertex of the out copy of $G$. Then
\[
(\lambda_\alpha(G)-\alpha(n_i-1+s)+1-n_i)x_i=sx_0=(\lambda_{\alpha}(G)-\alpha(n_{i+1}-1+s)+1-n_{i+1})x_{i+1}
\]
for $i=1,\dots, t-1$.
As $n_i\le n_{i+1}$ and $\lambda_{\alpha}(G)> \lambda_{\alpha}(K_{n_{i+1}})=(\alpha+1)(n_{i+1}-1)$, one gets
$x_i\le x_{i+1}$.
\end{proof}

Given a graph $G$,  $o(G)$  the number of odd components of  $G$, and $i(G)$  the number of isolated vertices of  $G$.

\section{Graphs with a matching and $1$-factor containing any given edge}

We need the following lemma.

\begin{Lemma}\label{matching} \cite{LGH}
Let $G$ be a graph on $n$ vertices. Let $k$ be an integer with $0\le k\le n$ and $k\equiv n \pmod 2$. Then $G$ has a matching of size $\frac{n-k}{2}$  containing any given edge if and only if
$o(G-S)\le |S|+k$
for all $S\subset V(G)$ and $o(G-S)= |S|+k$ implies that $S$ is an independent set.
\end{Lemma}

\subsection{Spectral radius}

\begin{Lemma} \label{h1} Let $n, s, k$ be positive integers with $n\ge \max\{5k+6,2s+k\}$, $s\ge 2$ and $0\le k\le n$. For fixed $n$ and $k$,
$\rho(K_s\vee ((s+k-1)K_1\cup K_{n-2s-k+1}))$ is uniquely maximized when $s=2$.
\end{Lemma}

\begin{proof} Let $H_s=K_s\vee ((s+k-1)K_1\cup K_{n-2s-k+1})$.
Partition $V(H_{s})$ into $S\cup V_1\cup V_2$, where $V_1=V((s+k-1)K_1)$ and $V_2=V(K_{n-2s-k+1})$.
It is easy to see that this partition is an equitable partition.  The corresponding quotient matrix is
 \[
\mathbf{B}=
 \begin{pmatrix}
     s-1 &  s+k-1 & n-2s-k+1 \\
     s & 0& 0\\
     s & 0& n-2s-k
 \end{pmatrix}.
 \]
By a simple calculation, the characteristic polynomial of $\mathbf{B}$ is
\[
f_s(x)=x^3-(n-k-s-1)x^2-(n+s^2+(k-2)s-k)x-2s^3+(n-3k+2)s^2+((k-1)n-k^2+k)s.
\]
By Lemma \ref{equit},  $\rho(H_{s})$ is equal to the largest root of the equation $f_s(x)=0$.
In particular, for $H_2=K_2\vee ((k+1)K_1\cup K_{n-k-3})$,
$\rho(H_2)$ is equal to the largest root of $f_2(x)=0$.

Suppose that $s>2$. We only need to show that $\rho(H_s)<\rho(H_2)$.
Note that
\[
\frac{f_s(x)-f_2(x)}{s-2}=h(x):=x^2-(s+k)x-2s^2+(n-3k-2)s+(k+1)n-k^2-5k-4.
\]
As  a quadratic function of $x$, $h(x)$ is strictly increasing for $x\ge x_0:=\frac{s+k}{2}$. As $n\ge  frac{5}{3}k+\frac{8}{3}$ and $n\ge 2s+k$, one gets $x_0\le n-k-2$.
Then, for $x\in[n-k-2,+\infty)$,   $h(x)$ is strictly increasing, so
\begin{align*}
h(x) & \ge  h(n-k-2)\\
&=-2s^2-2ks+n^2-(2k+3)n+k^2+k\\
&\ge  -\frac{(n-k)^2}{2}-k(n-k)+n^2-(2k+3)n+k^2+k\\
&=\frac{1}{2}n^2-(2k+3)n+\frac{3}{2}k^2+k\\
&\ge  \frac{1}{2}(5k+6)^2-(2k+3)(5k+6)+\frac{3}{2}k^2+k\\
&=4k^2+4k\\
&\ge 0,
\end{align*}
where the second inequality follows because as  a quadratic function of $s$, $h(n-k-2)$ is strictly decreasing for $2\le s\le \frac{n-k}{2}$, and the third inequality follows because
 as  a quadratic function of $n$, $\frac{1}{2}n^2-(2k+3)n-\frac{1}{2}k^2+k$ is strictly increasing for
 $n\ge 5k+6>2k+3$. Thus, $h(x)> 0$ for $x\in[n-k-2,+\infty)$, as $h(x)=0$ implies $k=0$, but then $h(x)\ge h(n-2)>0$, a contradiction.
 That is, $f_s(x)> f_2(x)$ for $x\in[n-k-2,+\infty)$.  By Lemma \ref{inequit}, we have
$\rho(H_2)> \rho(K_{n-k-1})=n-k-2$,
so $f_s(x)> f_2(x)\ge 0$ for $x\in[\rho(H_2),+\infty)$.
implying that $\rho(H_s)<\rho(H_2)$.
\end{proof}

Now we are ready to prove Theorem  \ref{Cai1}.

\begin{proof}[Proof of Theorem  \ref{Cai1}]
Suppose by contradiction that $G$ does not have a matching of size $\frac{n-k}{2}$  containing any given edge. By Lemma \ref{matching}, there exists a vertex subset $S\subset V(G)$ with $|S|=s$ such that
either $o(G-S)\ge  s+k$ or $G[S]$ is empty and  $o(G-S)> s+k$. Evidently, $n\equiv s+o(G-S) \pmod 2$, i.e., $o(G-S)\equiv n-s \pmod 2$.  As  $n\equiv k \pmod 2$, one gets $o(G-S)\equiv  k-s \equiv s+k\pmod 2$. Thus, if $G[S]$ is empty, then $o(G-S)\ge s+k+2$.
Let $t=o(G-S)$, $n_1\le \dots \le n_{t-1}$ be the orders
the odd components of $G-S$ with the first $t-1$ smallest orders, and let $n_t=n-s-n_1-\dots-n_{t-1}$.
Then $G$ is a spanning subgraph of the graph $G':=G[S]\vee (K_{n_1}\cup \dots \cup K_{n_t})$.
By Lemma \ref{inequit},
\[
\rho(G)\le  \rho(G')
\]
with equality when $S\ne \emptyset$ if and only if $G\cong G'$.
Let $H_s=K_s\vee((s+k-1)K_1\cup K_{n-2s-k+1})$.

\noindent
{\bf Case 1. }$G[S]$ is not empty.

In this case, $s\ge  2$ and  $n\ge  s+o(G-S)\ge  2s+k$. So $2\le  s\le  \frac{n-k}{2}$. By Lemma \ref{inequit} again,
\[
\rho(G)\le  \rho(G^*)
\]
with equality if and only if $G\cong G^*$, where $G^*:=K_s\vee (K_{n_1}\cup \dots \cup K_{n_t})$.
By Lemmas \ref{GE} and \ref{SP}, we have
\[
\rho(G^*)\le \rho(H_s)
\]
with equality if and only if $G^*\cong H_s$. Thus
\[
\rho(G)\le \rho(H_s)
\]
with equality if and only if $G\cong H_s$.
By the assumption, $\rho(G)\ge  \rho(H_2)$. Thus
\[
\rho(H_2)\le \rho(G) \le \rho(H_s)
\]
By Lemma \ref{h1}, $\rho(G)=\rho(H_2)$, so $G\cong H_2$, which is a contradiction.

\noindent
{\bf Case 2. }$G[S]$ is empty.

Suppose first that $s=0$. Then $t\ge  k+2$ and $\rho(G)\le  \rho(G')\le \rho(K_{n_t})\le  n-k-2$. Since $\rho(H_2)>\rho(K_{n-k-1})=n-k-2$, we have $\rho(G)\le \rho(K_{n_t})<\rho(H_2)$, which is a contradiction.

Suppose next that $s=1$. Then $\rho(G)\le \rho(K_1\vee (K_{n_1}\cup \dots \cup K_{n_{k+3}}))$. By Lemmas \ref{GE} and \ref{SP}, we have
\[
\rho(K_1\vee (K_{n_1}\cup \dots \cup K_{n_{k+3}}))\le \rho(K_1\vee ((k+2)K_1\cup K_{n-k-3})).
\]
By Lemma \ref{inequit}, we have
\[
\rho(K_1\vee ((k+2)K_1\cup K_{n-k-3}))< \rho(H_2).
\]
Thus, we have $\rho(G)<\rho(H_2)$, a contradiction.

Now, suppose that $s\ge 2$. By Lemmas \ref{GE} and \ref{SP}, we have
\[
\rho(G')\le \rho(sK_1\vee((s+k+1)K_1\cup K_{n-2s-k-1}))
\]
with equality if and only if $G'\cong sK_1\vee((s+k+1)K_1\cup K_{n-2s-k-1})$.
By Lemma \ref{inequit},
$\rho(sK_1\vee((s+k+1)K_1\cup K_{n-2s-k-1}))<\rho(H_s)$. It thus follows that
$\rho(G)< \rho(H_s)$.
By Lemma \ref{h1},
\[
\rho(G)\le \rho(H_2).
\]
Thus  $\rho(G)<\rho(H_2)$, which is a contradiction.
\end{proof}

As an immediate consequence of Theorem \ref{Cai1}, we have

\begin{Corollary}\label{c1}Let $G$ be a graph of order $n\ge  6$, where $n \equiv 0 \pmod 2$. If $\rho(G)\ge \rho(K_2\vee (K_1\cup K_{n-3}))$, then $G$ has a $1$-factor containing any given edge unless $G\cong K_2\vee (K_1\cup K_{n-3})$.
\end{Corollary}

\subsection{Signless Laplacian spectral radius}

\begin{Lemma} \label{h2} Let $n, s, k$ be positive integers with $n\ge \max\{5k+7,2s+k\}$, $s\ge 2$ and $0\le k\le n$. For fixed $n$ and $k$,
$q(K_s\vee ((s+k-1)K_1\cup K_{n-2s-k+1}))$ is uniquely maximized when $s=2$.
\end{Lemma}

\begin{proof} Let $H_s=K_s\vee ((s+k-1)K_1\cup K_{n-2s-k+1})$.
Partition $V(H_{s})$ into $S\cup V_1\cup V_2$, where $V_1=V((s+k-1)K_1)$ and $V_2=V(K_{n-2s-k+1})$.
It is easy to see that this partition is an equitable partition.  The corresponding quotient matrix is
 \[
\mathbf{B}=
 \begin{pmatrix}
     n+s-2 &  s+k-1 & n-2s-k+1 \\
     s & s& 0\\
     s & 0& 2n-3s-2k
 \end{pmatrix}.
 \]
By a simple calculation, the characteristic polynomial of $\mathbf{B}$ is
\begin{align*}
f_s(x)&=x^3-(3n-2k-s-2)x^2-(4s^2-(n-4k+4)s-2n^2+2(k+2)n-4k)x\\
&\quad -2s^3+(4n-4k-2)s^2-(2n^2-2(2k+1)n+2k^2+2k)s.
\end{align*}
By Lemma \ref{equit},  $q(H_{s})$ is equal to the largest root of the equation $f_s(x)=0$.
Similarly, let $H_2=K_2\vee ((k+1)K_1\cup K_{n-k-3})$.
$q(H_2)$ is equal to the largest root of $f_2(x)=0$.

Suppose that $s>2$.
It is easy to see that
\begin{align*}
\frac{f_s(x)-f_2(x)}{s-2}&=h(x):=x^2+(n-4s-4k-4)x-2s^2+(4n-4k-6)s-2n^2\\
&\quad +(4k+10)n-2k^2-10k-12.
\end{align*}
Evidently, $h(x)$ is strictly increasing for $x\ge x_0:=\frac{4s-n+4k+4}{2}$. As $n\ge \max\{2k+4,2s+k\}$, one gets
$x_0\le 2n-2k-4$.
It then follows  that  $h(x)$ is strictly increasing for $x\in[2n-2k-4,+\infty)$.
Thus, for $x\in[2n-2k-4,+\infty)$, one gets
\begin{align*}
h(x) & \ge  h(2n-2k-4)\\
&=-2s^2-(4n-4k-10)s+4n^2-(14k+18)n+10k^2+30k+20\\
&\ge  -\frac{(n-k)^2}{2}-\frac{(4n-4k-10)(n-k)}{2}+4n^2-(14k+18)n+10k^2+30k+20\\
&=\frac{3}{2}n^2-(9k+13)n+\frac{15}{2}k^2+25k+20\\
&\ge  \frac{3}{2}(5k+7)^2-(9k+13)(5k+7)+\frac{15}{2}k^2+25k+20\\
&=2k+\frac{5}{2}\\
&>0,
\end{align*}
where the second inequality follows because $h(2n-2k-4)$ is strictly decreasing for $2\le s \le \frac{n-k}{2}$, and the third inequality follows because
$\frac{3}{2}n^2-(9k+13)n+\frac{15}{2}k^2+25k+20$ is strictly increasing
for  $n\ge 5k+7>3k+\frac{13}{3}$.
Thus $f_s(x)> f_2(x)$ for $x\in[2n-2k-4,+\infty)$.  By Lemma \ref{inequit}, we have
$q(H_2)> q(K_{n-k-1})=2n-2k-4$,
so $f_s(x)> f_2(x)\ge 0$ for $x\in[q(H_2),+\infty)$.
implying that $q(H_s)< q(H_2)$.
\end{proof}

Now we are ready to prove Theorem  \ref{Cai2}.

\begin{proof}[Proof of Theorem  \ref{Cai2}]
Suppose by contradiction that $G$ does not have a matching of size $\frac{n-k}{2}$  containing any given edge. By Lemma \ref{matching}, there exists a vertex subset $S\subset V(G)$ with $|S|=s$ such that
either $o(G-S)\ge  s+k$ or $G[S]$ is empty and  $o(G-S)> s+k$. Evidently, $n\equiv s+o(G-S) \pmod 2$, i.e., $o(G-S)\equiv n-s \pmod 2$.  As  $n\equiv k \pmod 2$, one gets $o(G-S)\equiv  k-s \equiv s+k\pmod 2$. Thus, if $G[S]$ is empty, then $o(G-S)\ge s+k+2$.
Let $t=o(G-S)$, $n_1\le \dots \le n_{t-1}$ be the orders
the odd components of $G-S$ with the first $t-1$ smallest orders, and let $n_t=n-s-n_1-\dots-n_{t-1}$.
Then $G$ is a spanning subgraph of the graph $G':=G[S]\vee (K_{n_1}\cup \dots \cup K_{n_t})$.
By Lemma \ref{inequit}, $q(G)\le  q(G')$
with equality when $S\ne \emptyset$ if and only if $G\cong G'$.
Let $H_s=K_s\vee((s+k-1)K_1\cup K_{n-2s-k+1})$.

\noindent
{\bf Case 1. }$G[S]$ is not empty.

In this case, $s\ge  2$ and  $n\ge  s+o(G-S)\ge  2s+k$. So $2\le  s\le  \frac{n-k}{2}$. By Lemma \ref{inequit} again,
$q(G)\le  q(G^*)$
with equality if and only if $G\cong G^*$, where $G^*:=K_s\vee (K_{n_1}\cup \dots \cup K_{n_t})$.
By Lemmas \ref{GE} and \ref{SP}, we have
\[
q(G^*)\le q(H_s)
\]
with equality if and only if $G^*\cong H_s$. Thus
\[
q(G)\le q(H_s)
\]
with equality if and only if $G\cong H_s$.
By the assumption, $q(G)\ge  q(H_2)$. Thus
\[
q(H_2)\le q(G) \le q(H_s)
\]
By Lemma \ref{h2}, $q(G)=q(H_2)$, so $G\cong H_2$, which is a contradiction.

\noindent
{\bf Case 2. }$G[S]$ is empty.

Suppose first that $s=0$. Then $t\ge  k+2$ and $q(G)\le  q(G')\le q(K_{n_t})\le  2n-2k-4$. Since $q(H_2)>q(K_{n-k-1})=2n-2k-4$, we have $q(G)\le  q(K_{n_t})<q(H_2)$, which is a contradiction.

Suppose next that $s=1$. Then $q(G)\le q(K_1\vee (K_{n_1}\cup \dots \cup K_{n_{k+3}}))$. By Lemmas \ref{GE} and \ref{SP}, we have
\[
q(K_1\vee (K_{n_1}\cup \dots \cup K_{n_{k+3}}))\le q(K_1\vee ((k+2)K_1\cup K_{n-k-3})).
\]
By Lemma \ref{inequit}, we have
\[
q(K_1\vee ((k+2)K_1\cup K_{n-k-3}))< q(H_2).
\]
Thus, we have $q(G)<q(H_2)$, a contradiction.

Finally, suppose that $s\ge 2$.  By Lemmas \ref{GE} and \ref{SP}, we have
\[
q(G')\le q(sK_1\vee((s+k+1)K_1\cup K_{n-2s-k-1}))
\]
with equality if and only if $G'\cong sK_1\vee((s+k+1)K_1\cup K_{n-2s-k-1})$.
By Lemma \ref{inequit},
$q(sK_1\vee((s+k+1)K_1\cup K_{n-2s-k-1}))<q(H_s)$. It thus follows that
$q(G)< q(H_s)$.
By Lemma \ref{h2},
\[
q(G)\le q(H_2).
\]
Thus  $q(G)<q(H_2)$, which is a contradiction.
\end{proof}

By Theorem \ref{Cai2}, we have

\begin{Corollary}\label{c1}Let $G$ be a graph of order $n\ge  7$, where $n \equiv 0 \pmod 2$. If $q(G)\ge  q(K_2\vee (K_1\cup K_{n-3}))$, then $G$ has a $1$-factor containing any given edge unless $G\cong K_2\vee (K_1\cup K_{n-3})$.
\end{Corollary}

\section{Star factor with a component  $K_{1,1}$ or $K_{1,2}$ containing any given edge}

We need the following lemma.

\begin{Lemma} \cite{CEK}\label{CEK}
Let $G$ be a graph and let $k$ be an integer, $k\ge 2$.  Then $G$ has an $S(k)$-factor with a component being  $K_{1,1}$ or $K_{1,2}$ containing any given edge if and only if, for any proper subset $S$ of $V(G)$,
\[
i(G-S)\le \begin{cases}
k|S| & \mbox{if $G[S]$ is empty,}\\
k|S|-2k+1 &  \mbox{otherwise.}
\end{cases}
\]
\end{Lemma}

\subsection{Spectral radius}
\begin{Lemma} \label{h3} Let $n, s, k$ be positive integers with $n\ge \max\{\frac{3}{2}k+5,(k+1)s+1\}$, $s\ge 1$ and $k\ge2$. For fixed $n$ and $k$,
$\rho(sK_1\vee ((ks+1)K_1\cup K_{n-(k+1)s-1}))$ is uniquely maximized when $s=1$.
\end{Lemma}

\begin{proof} Let $G_s=sK_1\vee ((ks+1)K_1\cup K_{n-(k+1)s-1})$.
Partition $V(G_{s})$ into $S\cup V_1\cup V_2$, where $V_1=V((ks+1)K_1)$ and $V_2=V(K_{n-(k+1)s-1})$.
It is easy to see that this partition is an equitable partition.  The corresponding quotient matrix is
 \[
\mathbf{B}=
 \begin{pmatrix}
     0 &  ks+1 & n-(k+1)s-1 \\
     s & 0& 0\\
     s & 0& n-(k+1)s-2
 \end{pmatrix}.
 \]
By a simple calculation, the characteristic polynomial of $\mathbf{B}$ is
\[
f_s(x)=x^3-(n-s-ks-2)x^2-(ns-s^2)x-(k^2+k)s^3+(kn-3k-1)s^2+(n-2)s.
\]
By Lemma \ref{equit},  $\rho(G_{s})$ is equal to the largest root of the equation $f_s(x)=0$.
In particular,
$\rho(G_1)$ is equal to the largest root of $f_1(x)=0$.

Suppose that $s>1$.
It is easy to see that
\begin{align*}
\frac{f_s(x)-f_1(x)}{s-1}&=h(x):=(k+1)x^2-(n-s-1)x-(k^2+k)s^2+(kn-k^2-4k-1)s\\
&\quad +(k+1)n-k^2-4k-3.
\end{align*}
Note that $h(x)$ is strictly increasing for $x\ge x_0:=\frac{n-s-1}{2(k+1)}$. As $n\ge \frac{2k^2+6k+2}{2k+1}$ and $s\ge1$, one gets
$x_0\le n-k-2$.
It then follows  that  $h(x)$ is strictly increasing for $x\in[n-k-2,+\infty)$.
Thus, one may check that, for $x\in[n-k-2,+\infty)$,
\begin{align*}
h(x) & \ge  h(n-k-2)\\
&=-(k^2+k)s^2+((k+1)n-k^2-5k-3)s+kn^2-(2k^2+4k)n+k^3+4k^2+3k-1\\
&\ge -(k^2+k)+((k+1)n-k^2-5k-3)+kn^2-(2k^2+4k)n+k^3+4k^2+3k-1\\
&=kn^2- (2k^2+3k-1)n+k^3+2k^2-3k-4\\
&\ge k\left(\frac{3}{2}k+5\right)^2- (2k^2+3k-1)\left(\frac{3}{2}k+5\right)+k^3+2k^2-3k-4\\
&=\frac{1}{4}k^3 + \frac{5}{2}k^2+\frac{17}{2}k + 1\\
&>0.
\end{align*}
Thus $f_s(x)> f_1(x)$ for $x\in[n-k-2,+\infty)$.  By Lemma \ref{inequit}, we have
$\rho(G_1)>\rho(K_{n-k-1})=n-k-2$,
so
$f_s(x)> f_1(x)\ge 0$ for $x\in[\rho(H_1),+\infty)$
implying that $\rho(G_s)< \rho(G_1)$.
\end{proof}

\begin{Lemma} \label{h4} Let $n, s, k$ be positive integers with $n\ge \max\{ k+7,(k+1)s-2k+2\}$, $s\ge 2$ and $k\ge2$. For fixed $n$ and $k$,
$\rho(K_s\vee ((ks-2k+2)K_1\cup K_{n-(k+1)s+2k-2}))$ is uniquely maximized when $s=2$.
\end{Lemma}

\begin{proof} Let $G_s=K_s\vee ((ks-2k+2)K_1\cup K_{n-(k+1)s+2k-2})$.
Partition $V(G_{s})$ into $S\cup V_1\cup V_2$, where $V_1=V((ks-2k+2)K_1)$ and $V_2=V(K_{n-(k+1)s+2k-2})$.

It is easy to see that this partition is an equitable partition.  The corresponding quotient matrix is
 \[
\mathbf{B}=
 \begin{pmatrix}
     s-1 &  ks-2k+2 & n-(k+1)s+2k-2 \\
     s & 0& 0\\
     s & 0& n-(k+1)s+2k-3
 \end{pmatrix}.
 \]
By a simple calculation, the characteristic polynomial of $\mathbf{B}$ is
\begin{align*}
f_s(x)&=x^3-(n+2k-ks-4)x^2-(n+ks^2-(3k-2)s+2k-3)x-(k^2+k)s^3\\
&\quad +(kn+4k^2-3k-2)s^2-(2(k-1)n+4k^2-10k+6)s.
\end{align*}
By Lemma \ref{equit},  $\rho(G_{s})$ is equal to the largest root of the equation $f_s(x)=0$.
Particularly,
$\rho(G_2)$ is equal to the largest root of $f_2(x)=0$.

Suppose that $s>2$.
It is easy to see that
\[
\frac{f_s(x)-f_2(x)}{s-2}=h(x):=kx^2-(ks-k+2)x-(k^2+k)s^2+(kn+2k^2-5k-2)s+2n-10.\\
\]
Note that $h(x)$ is is strictly increasing for $x\ge x_0:=\frac{ks-k+2}{2k}$. As $n\ge \max\{\frac{7k^2+5k+2}{2k^2+k},(k+1)s+1\}$, one gets
$x_0\le n-3$.
Thus, for $x\in[n-3,+\infty)$,   we have
\begin{align*}
h(x) & \ge  h(n-3)\\
&=-(k^2+k)s^2+(2k^2-2k-2)s+kn^2-5kn+6k-4\\
&\ge  -(k^2+k)\left(\frac{n+2k-2}{k+1}\right)^2+(2k^2-2k-2)\frac{n+2k-2}{k+1}+kn^2-5kn+6k-4\\
&=\frac{k^2}{k+1}n^2- \frac{7k^2+3k+2}{k+1}n+\frac{6k^2-2k}{k+1}\\
& \ge  \frac{k^2}{k+1}(k+7)^2- \frac{7k^2+3k+2}{k+1}(k+7)+\frac{6k^2-2k}{k+1}\\
&= \frac{1}{k+1}(k^4+7k^3+3k^2-25k-14)\\
&>0.
\end{align*}
Thus $f_s(x)> f_2(x)$ for $x\in[n-3,+\infty)$.  By Lemma \ref{inequit}, we have
$\rho(G_2)> \rho(K_{n-2})=n-3$,
so  $f_s(x)> f_2(x)\ge 0$ for $x\in[\rho(G_2),+\infty)$,
implying that $\rho(G_s)\le \rho(G_2)$.
\end{proof}

Now we are ready to prove Theorem  \ref{Cai3}.

\begin{proof}[Proof of Theorem  \ref{Cai3}]
Suppose by contradiction that $G$ does not have an $S(k)$-factor containing any given edge. By Lemma \ref{CEK}, there exists a vertex subset $S\subset V(G)$ with $|S|=s$ such that
\[
i(G-S)\ge \begin{cases}
ks+1 & \mbox{if $G[S]$ is empty,}\\
ks-2k+2 &  \mbox{otherwise.}
\end{cases}
\]
As $G$ has no isolated vertices, $s\ge 1$. By Lemma \ref{inequit},
$\rho(G)\le  \rho(G_s)$
with equality if and only if $G\cong G_s$, where
\[
G_s:=\begin{cases}
sK_1\vee ((ks+1)K_1\cup K_{n-(k+1)s-1}) & \mbox{if $G[S]$ is empty,}\\
K_s\vee ((ks-2k+2)K_1\cup K_{n-(k+1)s+2k-2}) &  \mbox{otherwise.}
\end{cases}
\]

Suppose first that $G[S]$ is empty.
It is evident that $s\ge 1$.
From  $n\ge  s+i(G-S)$, we have $n\ge  (k+1)s+1$. So $2\le  s\le  \frac{n-1}{k+1}$.
By Lemmas \ref{h3} and \ref{inequit}, $\rho(G)\le \rho(G_s)\le \rho(G_1)<\rho(G_2)$, contracting
the assumption.

Suppose next $G[S]$ is not empty.
 Then $s\ge  2$. From $n\ge  s+i(G-S)$, we have $n\ge  (k+1)s-2k+2$. Thus $2\le  s\le  \frac{n+2k-2}{k+1}$.
By Lemma \ref{h4}, $\rho(G)\le \rho(G_s)\le \rho(G_2)$,
By the assumption, $\rho(G)\ge \rho(G_2)$. Thus $\rho(G)=\rho(G_2)$, implying that $G\cong G_2$,
a contradiction.
\end{proof}

\subsection{Signless Laplacian spectral radius}

\begin{Lemma} \label{h5} Let $n, s, k$ be positive integers with $n\ge \max\{2k+4,(k+1)s+1\}$, $s\ge 1$ and $k\ge2$. For fixed $n$ and $k$,
$q(sK_1\vee ((ks+1)K_1\cup K_{n-(k+1)s-1}))$ is uniquely maximized when $s=1$.
\end{Lemma}

\begin{proof} Let $G_s=sK_1\vee ((ks+1)K_1\cup K_{n-(k+1)s-1})$.
Partition $V(G_{s})$ into $S\cup V_1\cup V_2$, where $V_1=V((ks+1)K_1)$ and $V_2=V(K_{n-(k+1)s-1})$.
It is easy to see that this partition is an equitable partition.  The corresponding quotient matrix is
 \[
\mathbf{B}=
 \begin{pmatrix}
     n-s &  ks+1 & n-(k+1)s-1 \\
     s & s& 0\\
     s & 0& 2n-(2k+1)s-4
 \end{pmatrix}.
 \]
By a simple calculation, the characteristic polynomial of $\mathbf{B}$ is
\begin{align*}
f_s(x)&=x^3-(3n-(2k+1)s-4)x^2+(2n^2-4n-(2k+1)ns)x-(2k^2+4k+2)s^3\\
&\quad +(4(k+1)n-6k-6)s^2-(2n^2-6n+4)s.
\end{align*}
By Lemma \ref{equit},  $q(G_{s})$ is equal to the largest root of the equation $f_s(x)=0$,
and
$q(G_1)$ is equal to the largest root of $f_1(x)=0$.

Suppose that $s>1$.
It is easy to see that
\begin{align*}
\frac{f_s(x)-f_1(x)}{s-1}&=h(x):=(2k+1)x^2-(2k+1)nx-(2k^2+4k+2)s^2\\
&\quad +(4kn+4n-2k^2-10k-8)s-2n^2+(4k+10)n-2k^2-10k-12.
\end{align*}
Note that  $h(x)$ is strictly increasing for $x\ge x_0:=\frac{n}{2}$. As $n\ge  \frac{4}{3}k+\frac{8}{3}$, one gets
$x_0\le 2n-2k-4$.
It then follows  that  $h(x)$ is strictly increasing for $x\in[2n-2k-4,+\infty)$.
Thus
\begin{align*}
h(x) & \ge  h(2n-2k-4)\\
&=-(2k^2+4k+2)s^2+(4(k+1)n-2k^2-10k-8)s+4kn^2-(12k^2+26k+2)n\\
&\quad +8k^3+34k^2+38k+4\\
&\ge -(2k^2+4k+2)+(4(k+1)n-2k^2-10k-8)+4kn^2-(12k^2+26k+2)n\\
&\quad +8k^3+34k^2+38k+4\\
&=4kn^2-(12k^2+22k-2)n+8k^3+30k^2+24k-6\\
& \ge  4k(2k+4)^2-(12k^2+22k-2)(2k+4)+8k^3+30k^2+24k-6\\
&= 2k^2+4k+2\\
&>0.
\end{align*}
Thus $f_s(x)> f_1(x)$ for $x\in[2n-2k-4,+\infty)$.  By Lemma \ref{inequit}, we have
$q(G_1)> q(K_{n-k-1})=2n-2k-4$,
so
$f_s(x)> f_1(x)\ge 0$ for $x\in[q(G_1),+\infty)$,
implying that $q(G_s) \le q(G_1)$.
\end{proof}

\begin{Lemma} \label{h6} Let $n, s, k$ be positive integers with $n\ge \max\{2k+6,(k+1)s-2k+2\}$, $s\ge 2$ and $k\ge2$. For fixed $n$ and $k$,
$q(K_s\vee ((ks-2k+2)K_1\cup K_{n-(k+1)s+2k-2}))$ is uniquely maximized when $s=2$.
\end{Lemma}

\begin{proof} Let $G_s=K_s\vee ((ks-2k+2)K_1\cup K_{n-(k+1)s+2k-2})$.
Partition $V(G_{s})$ into $S\cup V_1\cup V_2$, where $V_1=V((ks-2k+2)K_1)$ and $V_2=V(K_{n-(k+1)s+2k-2})$.
It is easy to see that this partition is an equitable partition.  The corresponding quotient matrix is
 \[
\mathbf{B}=
 \begin{pmatrix}
     n+s-2 &  ks-2k+2 & n-(k+1)s+2k-2 \\
     s & s& 0\\
     s & 0& 2n-(2k+1)s+4k-6
 \end{pmatrix}.
 \]
By a simple calculation, the characteristic polynomial of $\mathbf{B}$ is
\begin{align*}
f_s(x)&=x^3-(3n-(2k-1)s+4k-8)x^2+(2n^2+(4k-10)n-4ks^2\\
&\quad -((2k-3)n-12k+12)s-8k+12)x-2k^2s^3+(4kn+8k^2-14k)s^2\\
&\quad -(2n^2+(8k- 14)n+8k^2-28k+24)s.
\end{align*}
By Lemma \ref{equit},  $q(G_{s})$ is equal to the largest root of the equation $f_s(x)=0$, and
$q(G_2)$ is equal to the largest root of $f_2(x)=0$.

Suppose that $s>2$.
It is easy to see that
\begin{align*}
\frac{f_s(x)-f_2(x)}{s-2}&=h(x):=(2k-1)x^2-((2k-3)n+4ks-4k+12)x-2k^2s^2\\
&\quad +(4kn+4k^2-14k)s-2n^2+14n-24.
\end{align*}
As $n\ge \max\{\frac{28k^2+12k}{6k^2+k-1},(k+1)s-2k+2\}$, one gets
$x_0:=\frac{(2k-3)n+4ks-4k+12}{2(2k-1)}\le 2n-6$.
So  $h(x)$ is strictly increasing for $x\in[2n-6,+\infty)$.
Thus
\begin{align*}
h(x) & \ge  h(2n-6)\\
&=-2k^2s^2-(4kn-4k^2-10k)s+4kn^2-(28k+4)n+48k+12\\
&\ge  -2k^2\left(\frac{n+2k-2}{k+1}\right)^2-(4kn-4k^2-10k)\frac{n+2k-2}{k+1}+4kn^2-(28k+4)n+48k+12\\
&=\frac{1}{(k+1)^2}\left((4k^3+2k^2)n^2-(40k^3+38k^2+18k+4)n+84k^3+92k^2+52k+12\right)\\
& \ge  \frac{1}{(k+1)^2}((4k^3+2k^2)(2k+6)^2-(40k^3+38k^2+18k+4)(2k+6)\\
&\quad +84k^3+92k^2+52k+12)\\
&= \frac{1}{(k+1)^2}(16k^5+24k^4-40k^3-100k^2-64k-12)\\
&=4(2k(2k^2-k-5)-3)\\
&>0.
\end{align*}
Thus $f_s(x)> f_2(x)$ for $x\in[2n-6,+\infty)$.  By Lemma \ref{inequit}, we have
$q(G_2)> q(K_{n-2})=2n-6$,
so
$f_s(x)> f_2(x)\ge 0$ for $x\in[q(G_2),+\infty)$,
implying that $q(G_s)\le q(G_2)$.
\end{proof}

Now we are ready to prove Theorem  \ref{Cai4}.

\begin{proof}[Proof of Theorem  \ref{Cai4}]
Suppose by contradiction that $G$ does not have an $S(k)$-factor containing any given edge. By Lemma \ref{CEK}, there exists a vertex subset $S\subset V(G)$ with $|S|=s$ such that
\[
i(G-S)\ge \begin{cases}
ks+1 & \mbox{if $G[S]$ is empty,}\\
ks-2k+2 &  \mbox{otherwise.}
\end{cases}
\]
As $G$ has no isolated vertices, $s\ge 1$. By Lemma \ref{inequit},
$q(G)\le  q(G_s)$
with equality if and only if $G\cong G_s$, where
\[
G_s:=\begin{cases}
sK_1\vee ((ks+1)K_1\cup K_{n-(k+1)s-1}) & \mbox{if $G[S]$ is empty,}\\
K_s\vee ((ks-2k+2)K_1\cup K_{n-(k+1)s+2k-2}) &  \mbox{otherwise.}
\end{cases}
\]

Suppose first that $G[S]$ is empty.
It is evident that $s\ge 1$.
From  $n\ge  s+i(G-S)$, we have $n\ge  (k+1)s+1$. So $2\le  s\le  \frac{n-1}{k+1}$.
By Lemmas \ref{h5} and \ref{inequit}, $q(G)\le q(G_s)\le q(G_1)<q(G_2)$, contracting
the assumption.

Suppose next $G[S]$ is not empty.
 Then $s\ge  2$. From $n\ge  s+i(G-S)$, we have $n\ge  (k+1)s-2k+2$. Thus $2\le  s\le  \frac{n+2k-2}{k+1}$.
By Lemma \ref{h6}, $q(G)\le q(G_s)\le q(G_2)$,
By the assumption, $q(G)\ge q(G_2)$. Thus $q(G)=q(G_2)$, implying that $G\cong G_2$,
a contradiction.
\end{proof}

\bigskip

\noindent {\bf Acknowledgement.}
This work was supported by National Natural Science Foundation of China (No. 12071158).

\end{document}